\documentclass{amsart}

\usepackage{amssymb}

\usepackage[all]{xy}

\newtheorem{thm}{Theorem}

\newtheorem{lem}[thm]{Lemma}
\newtheorem{cor}[thm]{Corollary}

\theoremstyle{definition}
\newtheorem{defn}[thm]{Definition}

\newtheorem{exmp}[thm]{Example}

\newtheorem{rem}[thm]{Remark}          

\newtheorem*{ack}{Acknowledgments}      

\newtheorem{defn-thm}[thm]{Definition--Theorem}  
\newtheorem{defn-lem}[thm]{Definition--Lemma}  

\theoremstyle{remark}

\renewcommand{\o}[0]{{\mathcal O}} 
\newcommand{\z}[0]{{\mathbb Z}}

\newcommand{\p}[0]{{\mathbb P}}
\newcommand{\f}[0]{{\mathbb F}}

\newcommand{\pic}[0]{\operatorname{Pic}}

\newcommand{\rank}[0]{\operatorname{rank}}

\newcommand{\supp}[0]{\operatorname{Supp}}

\newcommand{\im}[0]{\operatorname{im}}

\newcommand{\ext}[0]{\operatorname{Ext}}

\newcommand{\cliff}[0]{\operatorname{Cliff}}
\newcommand{\gon}[0]{\operatorname{gon}}
\newcommand{\spl}[0]{\operatorname{Spl}}

\newcommand{\sF}{\mathcal{F}}
\newcommand{\sG}{\mathcal{G}}

\def\loccoh#1.#2.#3.#4.{H^{#1}_{#2}(#3,#4)}

\DeclareMathAlphabet{\mathchanc}{OT1}{pzc}%
                                {m}{it}

\begin{document}
\bibliographystyle{amsalpha}


\title[Ulrich bundles on rational surfaces with an anticanonical pencil]{Ulrich bundles on rational surfaces with an anticanonical pencil}
\author{Yeongrak Kim}

\maketitle

\begin{abstract}
Ulrich bundles are the simplest sheaves from the viewpoint of cohomology tables. Eisenbud and Schreyer conjectured that every projective variety carries an Ulrich bundle, which means it has the same cone of cohomology table as the projective space of same dimension. In this paper we show the existence of stable rank 2 Ulrich bundle on rational surfaces with an anticanonical pencil, under a mild Brill-Noether assumption by using Lazarsfeld-Mukai bundles. Also we see that each of those surfaces carries its Chow form given by the Pfaffian of skew-symmetric morphism coming from an Ulrich bundle.
\end{abstract}

\section{Introduction}

Let $X \subset \p^N$ be a smooth projective variety of dimension $n$, and let $\sF$ be a coherent sheaf on $X$. The cohomology table of $\sF$ is given as the following data

\begin{center}
\begin{tabular}[c]{c c c c c}
\\
\hline
$\cdots$ & $\gamma_{n,-n-1}$ & $\gamma_{n, -n}$ & $\gamma_{n, -n+1}$ & $\cdots$ \\
$\cdots$ & $\gamma_{n-1,-n}$ & $\gamma_{n-1, -n+1}$ & $\gamma_{n-1, -n+2}$ & $\cdots$ \\
$\cdots$ & $\vdots$ & $\vdots$ & $\vdots$ & $\cdots$ \\
$\cdots$ & $\gamma_{0,-1}$ & $\gamma_{0,0}$ & $\gamma_{0,1}$ & $\cdots$ \\
\hline
\\
\end{tabular}
\end{center}
where $\gamma_{i,j} = h^i (X, \sF (j))$. Varying the coherent sheaves, these data define the cone of cohomology tables $\mathcal{C}(X, \o_X(1)) \subset \mathcal{C}(\p^N , \o_{p^N}(1))$. One natural question is: how to understand this cone? For instance, a (general) linear projection $\pi : X \to \p^n$ induces an injection $\pi_{*} : \mathcal{C} (X, \o_X(1)) \to \mathcal{C}(\p^n , \o_{p^n}(1))$. If there is a vector bundle $E$ on $X$ such that $\pi_{*}E$ is trivial, then $\pi^{*}( \cdot ) \otimes E$ will be an inverse of $\pi_{*}$ (multiply some rational number, if necessary), hence $\pi_{*}$ becomes an isomorphism. Such $E$ has the simplest cohomology table, and we call it \emph{Ulrich}. 

Eisenbud and Schreyer studied on those Ulrich bundles widely and made a conjecture so that every variety has an Ulrich bundle(\cite{ESW03}). They also proved the existence of Ulrich bundles for curves in their paper and the existence of rank 2 Ulrich bundles for del Pezzo surfaces. However, even for surfaces, Ulrich bundles are not well understood; only known results for cubic surfaces(\cite{CH11}), del Pezzo surfaces $X_d \subset \p^d$ of degree $d$ (\cite{CKM13}, \cite{MRPL14}), quartic surfaces(\cite{CKM12}), and for general K3 surfaces (\cite{AFO12}).

The key idea of the proof in \cite{AFO12} comes from Brill-Noether theory on K3 surfaces. They found that Lazarsfeld-Mukai bundles are indeed Ulrich with assuming mild Brill-Noether property. On the rational surfaces, the behavior of Lazarsfeld-Mukai bundles is also well-known and it is similar to the K3 case. For more details, see \cite{LC13}. 
\\

Let $(S, H=\o_S(1))$ be a polarized smooth rational surface of degree $d=H^2$ with an anticanonical pencil $h^0(\omega_S^{\vee}) \ge 2$, and $C \in |\omega_S (3)|$ be a general curve. Suppose that $C$ has genus $g \ge 4$, gonality $k=\gon(C)$, and its Clifford dimension is $1$. Our main result is:

\begin{thm}\label{Main}
Let $(S,H)$ be a polarized rational surface as above. Suppose furthermore that the Clifford index of $C$ is computed by the restriction of $\omega_S (1)$ on $C$. Then $S$ carries a $(d - K_S^2 + 5)$-dimensional family of stable rank 2 Ulrich bundles $E$ with $\det(E)=\omega_S (3)$.
\end{thm}

In particular, $d > (-K_S \cdot H) + 1$, which is equivalent to the condition $h^0 (C, \o_C(K_S +        H )) \ge 2$ and $h^1(C, \o_C (K_S + H)) \ge 2$. Our strategy is almost same as the proof in  \cite{AFO12}; showing that a certain class of Lazarsfeld-Mukai bundles satisfies the Ulrich condition. The difference comes from little adjustment while applying the analogue vector bundle techniques on rational surfaces, rather than K3 surfaces. In particular, the existence of anticanonical pencil is crucial in several places, which implies very nice properties of Lazarsfeld-Mukai bundles. 

Here the assumption $\cliff(C) = \cliff(\omega_S (1) \big{|}_C)$ assures the existence of a complete and base point free pencil of degree $\left( \frac{5}{2}d + \frac{3}{2}(K_S \cdot H) + 2 \right)$. It is obvious to show that $\cliff(\omega_S (1) \big{|}_C)$ is minimal among the Clifford indices of restrictions of adjoint line bundles $\cliff(\omega_S (m) \big{|}_C)$ for $m \in \z_+$. Nevertheless it is still a subtle question whether $\omega_S (1) \big{|}_C$ computes the Clifford index of $C$. 

\begin{rem}
We do not need to worry about the case $(d-K_S^2 + 5) \le 0$. This only happens for a projective plane $\p^2$, a quadric surface $Q=\p^1 \times \p^1 \subset \p^3$, and a blow-up of $\p^2$ at a point, which is a Hirzebruch surface $\f_1 \subset \p^4$. Each surface carries an Ulrich line bundle $\o_{\p^2}$, $L$ of type $(0,1)$, and $M$ of type $(1,1)$, respectively.
\end{rem}

\begin{ack}
The author thanks Marian Aprodu for very helpful discussions, suggestions and providing nice examples. This work was supported by NRF(National Research Foundation of Korea) Grant funded by the Korean Government (NRF-2010-Fostering Core Leaders of the Future Basic Science Program). 
\end{ack}

\section{Ulrich bundles on rational surfaces}

Let $X \subset \p^N$ be a smooth projective variety of dimension $n$. A vector bundle $E$ on $X$ is \emph{Ulrich} if and only if
\begin{eqnarray}
\ H^i (X, E(-i)) = 0 \text{ for } i>0 \text{ and } H^i (X, E(-i-1))=0 \text{ for } i<n.
\end{eqnarray}

For smooth surfaces $S$, the above condition is equivalent to the vanishing of the following cohomology groups 
\begin{displaymath}
H^0(S, E(-1)), H^1(S, E(-1)), H^1(S, E(-2)), H^2(S, E(-2)).
\end{displaymath}

This implies the further vanishing $H^0 (S, E(-2)) = H^2 (S,E(-1))=0$, so $E$ should be $0$-regular and $\chi (X,E(-1)) = \chi (X,E(-2))=0$. Applying Riemann-Roch to both of them and taking the difference, we obtain the relation 
\begin{eqnarray}
H \cdot \left( c_1 (E) - \frac{\rank(E)}{2}(K_S + 3H) \right) = 0.
\end{eqnarray}

This motivates the following definition of special Ulrich bundles by Eisenbud and Schreyer:

\begin{defn}[\cite{ESW03}]
A \emph{special Ulrich bundle} on a surface $S \subset \p^N$ is a $0$-regular rank $2$ vector bundle $E$ with determinant $\det (E) = \omega_S (3)$.  
\end{defn}

They also proved in \cite{ESW03} that if there is a rank 2 Ulrich bundle $E$ of $\det(E)=\omega_S(3)$, it must be a Lazarsfeld-Mukai bundle for some curve $C \in |\omega_S(3)|$ and line bundle $A \in \pic(C)$ of degree $c_2 (E)$. Hence it is natural to ask whether certain Lazarsfeld-Mukai bundles are Ulrich. 

\begin{rem}
A special Ulrich bundle is Ulrich. Let $E$ be a special Ulrich bundle on $S$. There is a perfect pairing $E \times E \to \det(E)$ so we have a natural isomorphism $E \simeq E^{\vee} \otimes \det(E)$.  By Serre duality, $H^0(S, E(-1)) \simeq H^2(S, E^{\vee}\otimes \omega_S (1))^{*} \simeq H^2(S, E(-2))^{*}$. Similarly, $H^1(S, E(-2)) \simeq H^1(S, E(-1))^{*}$.
\end{rem}

For technical reasons we need to assume the existence of anticanonical pencil $h^0 ( \omega_S ^{\vee}) \ge 2$. This guarantees the vanishing of $H^1$ of Lazarsfeld-Mukai bundles. Let $(S,H)$ be a polarized smooth, projective rational surface of degree $H^2 = d$ with an anticanonical pencil.

\begin{lem} \label{Pencil}
Let $(S,H)$ be a polarized smooth rational surface with an anticanonical pencil and $C \in |\omega_S (3)|$ a general curve. Suppose that $C$ has genus $g \ge 4$ and Clifford dimension $1$. Then $C$ carries a complete, base point free pencil $\mathfrak{g}^{1}_{g-k+3}$.
\end{lem}

\begin{proof}
Suppose that $C$ does not have a maximal gonality, say $k=\gon(C) \le \frac{g+2}{2}$. 
Then $C$ satisfies the linear growth condition
\[
\dim W_{\ell}^1 (C) \le \ell-k \text{ for } k \le \ell \le g-k+2.
\]
thanks to the results of \cite{LC13}. In particular, every component of $W_{g-k+2}^1 (C)$ has dimension $g-2k+2 = \rho(g,1,g-k+2) \ge 0$. Then it follows that each component of $W_{g-k+3}^1 (C)$ has dimension $\rho(g,1,g-k+3) = g-2k+4$ via excess linear series, where $\dim (C + W_{g-k+2}^1 (C)) = g-2k+3$. Hence we conclude that the general element in every component of $W_{g-k+3}^1 (C)$ is base point free and complete. 

For $C$ having odd genus and maximal gonality, it has Clifford dimension 1 automatically(cf. \cite{Apr13}) and carries a $\infty^1$-family of minimal pencil. Here, $g-k+3=k$ and it cannot have any $\mathfrak{g}^1_{g-k+2}$ so the general element in $W_{g-k+3}^1 (C)$ is complete and base point free.
\end{proof}

\begin{rem}
In \cite{LC13}, the author also classified all the cases occuring exceptional curves(having Clifford dimension greater than 1). For some partial results on blow-ups of projective plane $\p^2$, see \cite{ESW03}.
\end{rem}

We need the following lemma for the proof of \emph{Theorem \ref{Main}}.

\begin{lem}\label{Cycles}
Let $(S,H)$ be a polarized smooth rational surface of degree $H^2 =d$ with an anticanonical pencil, and let $D \in |\omega_S (2)|$ be a smooth section. Then
\begin{displaymath}
\dim \{ \Gamma \in D_{\left(\frac{5}{2} d + \frac{3}{2}(K_S \cdot H) +2 \right)} \ | \ 
h^0 (D, \o_D (\Gamma - H )) \ge 1 \} \le d - (K_S \cdot H) - K_S^2 + 3.
\end{displaymath}
\end{lem}

\begin{proof}
Direct calculations show that $\deg (\o_D (H)) = 2d + (K_S \cdot H)$, $g(D) = 1 + \frac{1}{2} D \cdot (D+K_S ) = 2d + 3(K_S \cdot H) + K_S ^2 + 1$, and $\omega_D \simeq \o_D (2K_S + 2H)$. For simplicity, let $\alpha = \deg (\Gamma) - \deg(\o_D(H))= \frac{1}{2}d + \frac{1}{2}(K_S \cdot H) + 2$. We consider the incidence variety
\[
\mathcal{V} := \left\{ (\Gamma, \zeta) \in D_{\left(\frac{5}{2} d + \frac{3}{2}(K_S \cdot H) +2 \right)} \times D_{\alpha} \ \Big|\  \Gamma \in |\o_D (H + \zeta)| \right\}
\]
together with projections $\pi_1 : \mathcal{V} \to D_{\left(\frac{5}{2} d + \frac{3}{2}(K_S \cdot H) +2 \right)}$ and $\pi_2 : \mathcal{V} \to D_{\alpha}$. Note that $\pi_1 (\mathcal{V})$ is the variety whose dimension we have to compute. To estimate $\dim (\mathcal{V})$ we first observe the fibers of $\pi_2$. Riemann-Roch says
\begin{eqnarray*}
h^0 (D, \o_D (H + \zeta)) &=& 1- g(D) + \deg( \o_D (H + \zeta )) + h^1 (D, \o_D (H + \zeta)) \\
& = & h^0 (D, \o_D (2K_S + H - \zeta)) + \left( \frac{1}{2} d - \frac{3}{2}(K_S \cdot H) - K_S^2 + 2 \right).
\end{eqnarray*}
for every $\zeta \in D_{\alpha}$. 

Also notice that $H^1 (S, \o_S (1)) \simeq H^1 (H, \o_H (1)) = 0$ since $\deg \o_H(1) = d = 2 g(H) - 2 - (K_S \cdot H)$ and $-(K_S \cdot H) > 0$, that is, $\o_H(1)$ is nonspecial. By Serre duality, $H^1 (S, \omega_S (-1)) \simeq H^1(S, \o_S(1))^{*}$ also vanishes. This implies $H^0 (D, \o_D (2K_S + H)) \simeq H^0 (S, \o_S (2K_S + H)) \subseteq H^0 (S, \o_S (K_S + H))$ and $h^0 (S, \o_S (K_S + H)) = \chi(S, \o_S(K_S + H)) = \frac{1}{2} d + \frac{1}{2} (K_S \cdot H) + 1$ by Kodaira vanishing. Applying geometric Riemann-Roch theorem, we have
\begin{eqnarray*}
h^0 (D, \o_D (2K_S + H - \zeta)) &=& \max [ 0, h^0 (D, \o_D (2K_S + H)) - \deg \zeta] \\
& \le & \max[0, h^0(S, \o_S (K_S + H)) - \deg \zeta] \\
& = & \max [0,  \frac{1}{2}d + \frac{1}{2} (K_S \cdot H) + 1 - \alpha] \\
& = & 0
\end{eqnarray*}
for general divisor $\zeta$. Since the fiber $\pi_2 ^{-1} (\zeta) = \p H^0 (D, \o_D (H + \zeta))$, $\mathcal{V}$ has a unique irreducible component of dimension
\begin{eqnarray*}
\dim \pi_2^{-1} (\zeta) + \dim D_{\alpha} 
&\le& \left( \frac{1}{2} d - \frac{3}{2}(K_S \cdot H) - K_S^2 + 2 \right) -1 + \alpha \\
&\le & d - (K_S \cdot H) - K_S^2 + 3.
\end{eqnarray*}

Also consider $\Sigma_i := \left\{ \zeta \in D_{\alpha} \ |\  h^0(D, \o_D (2K_S + H - \zeta )) = i \right\}$, the locally closed variety of $D_\alpha$, which has dimension at most $\dim |\o_D(2K_S + H)| - i + 1 \le \frac{1}{2}d + \frac{1}{2}(K_S \cdot H) - i + 1$, for $i \ge 1$. For $\zeta \in \Sigma_i$, the fiber has larger dimension $\pi_2^{-1}(\zeta) \simeq \p^{\left( \frac{1}{2}d - \frac{3}{2}(K_S \cdot H) - K_S^2 + i + 1 \right)}$ than the general case, hence 
\begin{eqnarray*}
\dim \pi_2^{-1} (\Sigma_i ) & \le&  \dim \Sigma_i + \dim \pi_2^{-1}(\zeta) \\
& \le & d - (K_S \cdot H) - K_S^2 +2.
\end{eqnarray*}
Together with the general situation, we conclude that all components of $\mathcal{V}$ are of dimension $\le d - (K_S \cdot H) - K_S^2 + 3$. 
\end{proof}

Note that the above lemma shows one of the technical difference with the K3 case. The existence of anticanonical pencil implies that there is an injection $\o_S(2 K_S + H) \hookrightarrow \o_S(K_S + H)$. This makes possible to get an appropriate bound, which is an useless process for K3 surfaces. With above lemmas, now we can prove our main theorem:

\begin{proof}[Proof of Theorem \ref{Main}]
We start with some computations of invariants for convenience. Let $C \in |\omega_S(3)|$ be a general curve satisfying all the conditions of the theorem, and fix it. By the Riemann-Roch and adjunction formula, we get

\begin{itemize}
\item $g = g(C) = \frac{9}{2}d + \frac{9}{2}(K_S \cdot H) + K_S^2 + 1$
\item $\cliff(C) = \cliff(\o_C (K_S + H)) = 2d + 3(K_S \cdot H) + K_S^2$
\item $k = \gon(C) = \cliff(C) + 2 = 2d + 3(K_S \cdot H) + K_S^2 + 2$
\item $g-k+3 = \frac{5}{2}d + \frac{3}{2}(K_S \cdot H) + 2$.
\end{itemize}

Note that the number $(g-k+3)$ is the degree of $0$-cycles $\Gamma$ in \emph{Lemma \ref{Cycles}}. According to \emph{Lemma \ref{Pencil}}, $C$ carries a complete and base point free pencil $A \in W^1_{g-k+3}$. As we discussed above, the Lazarsfeld-Mukai bundle $E_{C,A}$ is a strong candidate to be Ulrich. Now we consider the relative Brill-Noether scheme $\mathcal{W}^1_{g-k+3} ( |\omega_S (3)| )$. It is clear that the general point $(C,A)$ in any dominating component $\mathcal{W} \subseteq \mathcal{W}^1_{g-k+3}$ over the linear system $|\omega_S (3)|$ corresponds to a complete and base point free pencil $\mathfrak{g}^1_{g-k+3}$. We may also assume that $A \not\simeq \omega_S^{\vee} \otimes \o_C$ which implies $h^0 (S, E_{C,A}) = \chi (S, E_{C,A}) = 2d$(cf. (3.1) of \cite{LC13}). We will abbreviate $E_{C,A}$ by $E$ if there is no confusion.

Since the Ulrich conditions are open conditions coming from the cohomology vanishing, it is natural to ask when a Lazarsfeld-Mukai bundle fails to be Ulrich. We need to verify that the \emph{non-Ulrich locus} cannot fill the whole $\mathcal{W}$, that is, the Lazarsfeld-Mukai bundle $E_{C,A}$ corresponding to a general point of $\mathcal{W}$ is indeed Ulrich. To do this, we first check the partial Ulrich condition 
\begin{equation}\label{cond:partialUlrich}
H^0 (S, E(-1)) = 0.
\end{equation}

We shall find an explicit parametrization of the failure locus of (\ref{cond:partialUlrich}) and count the number of parameters. Let $\sG = \left\{(E_{C,A},  \Lambda) \ \big{|}\ (C,A) \in \mathcal{W}, \Lambda \in G(2, H^0 (S, E_{C,A})) \right\} $ be the Grassmannian bundle over the moduli space of Lazarsfeld-Mukai bundles. We have a lower bound for the dimension

\begin{eqnarray*}
\dim \mathcal{W} & \ge& \dim |\omega_S(3)| + \rho (g, 1 g-k+3) \\
& = & \chi(\omega(3)) - 1 + \rho(g,1,g-k+3) \\
& = & 5d - K_S^2 + 1.
\end{eqnarray*}
Here the projection $\sG \to \mathcal{W}$ is dominant (the fiber is characterized by the group $H^2 (S, E_{C,A} \otimes E_{C,A}^{\vee} )$, cf. \cite{LC13}), we get a same lower bound $\dim \sG \ge \dim \mathcal{W} \ge 5d - K_S^2 + 1$. Since $h^0 (E) = 2d$, the dimension of the space of Lazarsfeld-Mukai bundles corresponding to the pairs $(C,A) \in \mathcal{W}$ has dimension at least $\dim \sG - \dim G(2, H^0 (S, E_{C,A})) \ge d-K_S^2 + 5$.

Next, we consider the projective bundle
\[
\mathcal{P} := \left\{ (E_{C,A}, \ell) \ \big{|}\  (C,A) \in \mathcal{W}, \  \ell \in \p H^0 (S, E_{C,A}) \right\}
\]
with $\dim \mathcal{P} \ge (d-K_S^2 + 5) + h^0 (S, E_{C,A}) - 1 = 3d - K_S^2 + 4$. This construction allows us to represent $E_{C,A}$ as an extension
\begin{equation} \label{seq:extension}
0\longrightarrow \mathcal O_S \stackrel{\ell} \longrightarrow E_{C,A}\longrightarrow \mathcal I_{\Gamma/S} \otimes \omega_S(3)\longrightarrow 0,
\end{equation}
where $\Gamma \in S^{[g-k+3]}$ is a 0-dimensional subscheme which satisfies the \emph{Cayley-Bacharach} condition(CB) with respect to $|\omega_S^2 \otimes \o_S (3)|$. For more details about CB, see \cite{GH78}, \cite{Cat90}, \cite{Laz97}. From the Grothendieck-Serre duality we have

\begin{eqnarray*}
\ext^1 (\mathcal{I}_{\Gamma/S} \otimes \omega_S (3), \o_S ) & \simeq & \ext^1 (\mathcal{I}_{\Gamma/S} \otimes \omega_S (3) \otimes \omega_S , \omega_S) \\
& \simeq & H^1 (S, \mathcal{I}_{\Gamma/S} \otimes \omega_S^2 \otimes \o_S (3))^{*}.
\end{eqnarray*}

Twisting the exact sequence (\ref{seq:extension}) by $\omega_S$ and taking the cohomology, we get an exact sequence
\[
0 = H^1 (S, E \otimes \omega_S ) \rightarrow H^1 (S, \mathcal{I}_{\Gamma/S} \otimes \omega_S^2  (3)) \rightarrow H^2(S, \omega_S ) \rightarrow H^2 (S, E \otimes \omega_S ) = 0,
\]
which says $\dim \ext^1(\mathcal{I}_{\Gamma/S} \otimes \omega_S (3), \o_S ) = h^2(S, \omega_S) = 1$. In particular, $\Gamma$  uniquely determines the Lazarsfeld-Mukai bundle $E_{C,A}$ and the map $\varphi : \mathcal{P} \to S^{[g-k+3]}$ given by $\varphi (( E_{C,A}, \ell)) := \Gamma$ which is generically injective onto its image.

Since $H^0 (S, E(-1)) \simeq H^0 (S, \mathcal{I}_{\Gamma / S} \otimes \omega_S(2))$, it is enough to show that the cycles $\Gamma \in \im(\varphi)$ with $H^0 (S, \mathcal{I}_{\Gamma / S} \otimes \omega_S(2)) \neq 0$ depend on at most $(3d - K_S^2 + 3) \le \dim \mathcal{P} -1$ parameters. We consider the incidence variety
\[
\mathcal{Z} = \{ (D, \Gamma) : D \in |\omega_S (2)|, \ \Gamma \subset D
\text{ satisfies CB with respect to } |\omega_S^2  (3)| \}.
\]
Choose and fix a smooth section $D \in |\omega_S(2) |$ and an effective divisor $\Gamma \in D_{g-k+3}$. For a point $p \in \supp(\Gamma)$, we write $\Gamma = \Gamma_p + p$ where $\Gamma_p \in D_{g-k+2}$. Since $H^1 (S, \omega_S (1))=0$ by Kodaira vanishing, we get a short exact sequence
\[
0 \longrightarrow H^0 (S, \omega_S (1)) \stackrel{+D} \longrightarrow H^0 (S,\mathcal{I}_{\Gamma/S} \otimes \omega_S^2 (3)) \longrightarrow H^0 (D, \o_D (2K_S + 3H - \Gamma)) \longrightarrow 0.
\]
We interpret the Cayley-Bacharach condition for $\Gamma$ as satisfying $H^0 (D, \o_D(2K_S + 3H - \Gamma)) \simeq H^0 (D, \o_D(2K_S + 3H - \Gamma_p))$ for any $p \in \supp(\Gamma)$, or equivalently by Riemann-Roch,
\[
h^0 (D, \o_D (\Gamma_p - H)) = h^0 (D, \o_D (\Gamma - H)) -1, \text{ for each } p \in \supp(\Gamma).
\]
In particular, $h^0 (D, \o_D (\Gamma - H)) \ge 1$. Applying \emph{Lemma \ref{Cycles}}, we conclude that the dimension of each fiber $\mathcal{Z} \to |\omega_S(2)|$ does not exceed $[d-(K_S \cdot H) - K_S^2 + 3]$; thus $\dim \mathcal{Z} \le \dim |\omega_S(2)| + [d-(K_S \cdot H) - K_S^2 + 3] = 3d - K_S^2 + 3$, which confirms that $E_{C,A}$ satisfies the partial Ulrich condition (\ref{cond:partialUlrich}) for a general $(C,A) \in \mathcal{W}$.

Indeed, these Lazarsfeld-Mukai bundles are Ulrich. By the Riemann-Roch formula, it is easily seen that 
\[
\chi(S, E(-1)) = \left( \frac{5}{2}d + \frac{3}{2} (K_S \cdot H) + 2 \right) - c_2 (E) = 0.
\]
Also note that $H^2 (S, E(-1)) \simeq  H^0 (S, E(-2))^{*} = 0$. Therefore, we conclude that $E$ is Ulrich.

Finally, we want to check the stability. In our case, the local dimension at $E_{C,A}$ of the moduli space $\spl(2; \omega_S(3), g-k+3)$ of simple vector bundles of rank 2 on $S$ with first Chern class $\omega_S (3)$ and second Chern class $g-k+3$ is computed from the dimension of $M_H (2; \omega_S (3), g-k+3)$, the moduli space of rank 2 $H$-stable vector bundles on $S$ with given Chern classes. Thanks to \cite{Don86}, \cite{Zuo91} and \cite{CMR99}, $M_H (2; \omega_S (3), g-k+3)$ is (nonempty) smooth and irreducible variety of dimension $4 c_2 (E) - c_1^2 (E) - 3 \chi (\o_X) = d - K_S^2 + 5$ and it lies as an open dense subset in $\spl(2; \omega_S (3), g-k+3)$. Since Ulrich bundles are already semistable, our dimension estimation shows that the locus of strictly semistable cannot fill up the whole moduli space(see also (4.2) of \cite{LC13}). This completes the proof.
\end{proof}

\begin{exmp}
Let $(S, H)$ be a pair of Hirzebruch surface $\f_2$ and a very ample line bundle of type $(2, 5)$. Note that $h^0 (\omega_S ^{\vee}) = \chi( \omega_S ^{\vee}) = 1 + K_S^2 = 9.$ Since a smooth curve $C \in |K_S + 3H|$ is of type $(4, 11)$ which is 4-gonal, it is easy to see that $\cliff(C) = 2$. 
Simple calculations show that the Clifford index for a line bundle $\o_C(K_S + H)$ which is of $(0,1)$-type, the ruling on $S$, is 2. This example satisfies our assumption. 
\end{exmp}

\begin{rem}
The hypotheses on $C$ in \emph{Theorem \ref{Main}} only used to ensure a Brill-Noether condition: C carries a base point free pencil of degree $\left( \frac{5}{2}d + \frac{3}{2}(K_S \cdot H) + 2\right)$. 
\end{rem}

\begin{rem}
It is well known that the Hirzebruch surface $\f_a$ embedded in the projective space by the ample line bundle of type $(1,n), n>a$ has an Ulrich line bundle of type $(0, 2n-a-1)$. On the other hand, if we consider the embeddings of type $(2,n)$, there is no Ulrich line bundle in general. Let $H=(2,n)$, $n>2a$ be a very ample line bundle on $\f_a$. Suppose that a line bundle $L = (s,t)$ is Ulrich. Taking difference of the Euler characteristic $\chi(\f_a, L-H)=0$ and $\chi(\f_a, L-2H)=0$ we have an identity $t = \left( a - \frac{1}{2}n \right) s + \frac{1}{2}(5n-5a-2)$. Hence $\chi(\f_a, L-H) = \frac{1}{2}(s-3)(s-1)(a-n) = 0$.

Consider the case $s=1$. Use the identification in \cite{FM11},  $H^1(\f_a, L-2H) = H^1 \left(\f_a, (-3, - \frac{3}{2}a - 1 ) \right) \simeq H^0 \left(\p^1, \o_{\p^1} ( - \frac{1}{2}a - 1)\right) \oplus H^0 \left( \p^1, \o_{\p^1} (\frac{1}{2}a - 1 ) \right)$ which cannot be $0$ unless $a=0$, so $L$ cannot be Ulrich. Similarly we can find a contradiction for the case $s=3$; $H^1( \f_a, L-H )$ cannot be $0$.

To sum up, the pair $(\f_a , H=(2,n))$ has no Ulrich line bundle for $a>0$. On the other hand, as seen in the previous example, it can easily seen that $(\f_a , H=(2, a+3))$ carries a rank 2 Ulrich bundle for $0 \le a \le 2$.  As an exceptional case, for $(\f_0 = \p^1 \times \p^1 , H=(2,n))$, it carries Ulrich line bundles of the type $(1, 2n-1)$ and $(3, n-1)$. 
\end{rem}

\section{The Chow form of rational surfaces}

Another application of our result is that the Chow form of a polarized rational surfaces which describe in the main theorem has a Pfaffian form in Pl{\"u}cker coordinates. Let $(S,H) \subset \p^{d/2 - (K_S \cdot H)/2}$ be a rational surface which satisfies the hypotheses of the theorem, and fix it. We introduce the exterior algebras $\bold{\Lambda}$ and $\bold{\Lambda}^{\vee}$, with gradings 
\[
\bold{\Lambda}_{-i} := \bigwedge^i H^0 (S, \o_S(1))^* \text{ and } \bold{\Lambda}^{\vee}_i := \bigwedge^i H^0 (S, \o_S(1))
\]
respectively. Now choose a special Ulrich bundle $E$ on $S$, considering as a sheaf on $\p^{d/2 - (K_S \cdot H)/2}$, we have the following Tate resolution
\[
T^{\bullet}(E) : \cdots \longrightarrow T^{-1}(E) \stackrel{\varphi_E} \longrightarrow T^0 (E) \longrightarrow T^1(E) \longrightarrow \cdots
\]
where $T^p(E)$ is isomorphic to $ \bigoplus_{i=0}^2 \bold{\Lambda}^{\vee} \otimes H^i (S, E(p-i))$ as graded $\bold{\Lambda}$-modules. Since $E$ is Ulrich, the resolution is quite simple, in particular, $T^{-1}(E) = \bold{\Lambda}^{\vee}(3) \otimes H^2 (S, E(-3))$ and $T^0(E) = \bold{\Lambda}^{\vee} \otimes H^0 (S, E)$. Following up \cite{ESW03}, we pass through the functor $\bold{U}_3$ from the category of free graded $\bold{\Lambda}$ modules to the category of vector bundles over $\bold{G} := G\left(\frac{1}{2}d - \frac{1}{2}(K_S \cdot H) - 2,  H^0 (S, \o_S(1))^{*}\right)$ the Grassmannian of codimension 3 planes in $\p^{d/2 - (K_S \cdot H)/2}$. The complex $\bold{U}_3 (E)$ is composed of a single morphism of vector bundles on $\bold{G}$
\[
\varphi : H^2 (S, E(-3)) \otimes \bigwedge^3 \mathcal U \to H^0 (S, E) \otimes \o_{\bold{G}}
\]
where $\mathcal{U}$ is the rank 3 tautological bundle on $\bold{G}$. Applying (3.4) of \cite{ESW03}, we get the following result:

\begin{cor}
Let $(S, H)$ and $C \in |\omega_S(3)|$ be a rational surface and a general curve described in \emph{Theorem \ref{Main}}. Let $E$ be a rank 2 Ulrich bundle on $S$ with $\det(E) = \omega_S (3)$. Then there exists a skew-symmetric morphism of vector bundles of rank $2d$ on $\bold{G}$
\[
\bold{U}_3 (\varphi_E) : H^2 (S, E(-3)) \otimes \bigwedge^3 \mathcal U \to H^0 (S, E) \otimes \o_{\bold{G}}
\]
whose Pfaffian is precisely the Chow form of $S$. 
\end{cor}

We skip the proof since the result itself is just a direct consequence of the existence of special Ulrich bundles(cf. (3.4) of \cite{ESW03}). Also it is worthwhile to see \cite{AFO12}, which provides an explicit description for this linear map. Since the result comes from properties of Ulrich bundles, exactly the same description holds for our case.
\\
\\

\def\cprime{$'$} \def\cprime{$'$} \def\cprime{$'$} \def\cprime{$'$}
  \def\cprime{$'$} \def\cprime{$'$} \def\dbar{\leavevmode\hbox to
  0pt{\hskip.2ex \accent"16\hss}d} \def\cprime{$'$} \def\cprime{$'$}
  \def\polhk#1{\setbox0=\hbox{#1}{\ooalign{\hidewidth
  \lower1.5ex\hbox{`}\hidewidth\crcr\unhbox0}}} \def\cprime{$'$}
  \def\cprime{$'$} \def\cprime{$'$} \def\cprime{$'$}
  \def\polhk#1{\setbox0=\hbox{#1}{\ooalign{\hidewidth
  \lower1.5ex\hbox{`}\hidewidth\crcr\unhbox0}}} \def\cdprime{$''$}
  \def\cprime{$'$} \def\cprime{$'$} \def\cprime{$'$} \def\cprime{$'$}
\providecommand{\bysame}{\leavevmode\hbox to3em{\hrulefill}\thinspace}
\providecommand{\MR}{\relax\ifhmode\unskip\space\fi MR }
\providecommand{\MRhref}[2]{%
  \href{http://www.ams.org/mathscinet-getitem?mr=#1}{#2}
}
\providecommand{\href}[2]{#2}

\vskip1cm

\noindent KAIST, Daejeon, Korea

{\begin{verbatim}Yeongrak.Kim@gmail.com\end{verbatim}}

\end{document}